\documentclass[12pt,letterpaper]{amsart}
\usepackage{amssymb,latexsym, amsmath}
\usepackage[dvips]{graphics}

\theoremstyle{plain}
\newtheorem{theorem}{Theorem}[section]

\newtheorem{conjecture}[theorem]{Conjecture}

\newtheorem{lemma}[theorem]{Lemma}
\newtheorem{proposition}[theorem]{Proposition}
\newtheorem{claim}[theorem]{Claim}
\theoremstyle{definition}

\theoremstyle{remark}

\numberwithin{equation}{section}
\numberwithin{theorem}{section}
\numberwithin{table}{section}
\numberwithin{figure}{section}

\newcommand{\R}{\mathbb R}

\def\({\left(}
\def\){\right)}
\def\logg{\log_{(2)}}

\begin{document}
\title[Siegel zeros and zeros of the derivative]{Landau-Siegel zeros and zeros of the derivative of the Riemann zeta function}
\author{David~W.~Farmer and Haseo Ki}

\thanks{
Research of the first author supported by the
American Institute of Mathematics, and the National Science Foundation.
Research of the second author
supported by the Korea Research Foundation Grant
funded by the Korean Government(MOEHRD, Basic Research Promotion
Fund)(KRF-2008-313-C00009).
}

\thispagestyle{empty}
\vspace{.5cm}
\begin{abstract}
We show that if
the derivative of the Riemann zeta function
has  sufficiently many zeros close to the critical line, then the
zeta function has many closely spaced zeros.  This gives a
condition on the zeros of the derivative of the zeta function
which implies a lower bound of the class numbers of imaginary
quadratic fields.
\end{abstract}

\address{
{\parskip 0pt
American Institute of Mathematics\endgraf
farmer@aimath.org\endgraf
\null
Department of Mathematics\endgraf
Yonsei University\endgraf
haseo@yonsei.ac.kr\endgraf
}
  }

\maketitle

\section{Introduction}

The spacing between zeros of the Riemann zeta-function and the location
of zeros of the derivative of the zeta-function are closely
related problems which have connections to other topics in 
number theory.

For example, 
if the zeta-function had a large number of pairs of zeros that were
separated by less than half their average spacing, one would obtain
an effective lower bound on the class numbers of imaginary
quadratic fields~\cite{M,CI}.
Also, Speiser proved that the Riemann hypothesis is equivalent
to the assertion that the
nontrivial zeros of the derivative of the zeta-function,
$\zeta'$, are
to the right of the critical line~\cite{Sp}.  There is a
quantitative version of Speiser's theorem~\cite{LM} which
is the basis for Levinson's method~\cite{L}.  In Levinson's
method there is a loss caused by the zeros of $\zeta'$ which
are close to the critical line, so it would be helpful to
understand the horizontal distribution of zeros of $\zeta'$.
The intuition is that the spacing of zeros of the zeta-function
should determine the horizontal distribution of zeros of
the derivative.  Specifically, a pair of closely spaced zeros
of $\zeta(s)$ gives rise to a zero of $\zeta'(s)$ close to the
critical line.  Our main result is a partial converse,
showing that \emph{sufficiently many} zeros of $\zeta'(s)$
close to the $\tfrac12$-line implies the existence of
many closely spaced zeros of~$\zeta(s)$.  See Theorem~\ref{t:la}.

We assume the Riemann hypothesis and write the zeros of
$\zeta$ as $\rho_j=\tfrac12+i\gamma_j$ and the zeros of
$\zeta'$ as $\beta_j'+i\gamma_j'$, where in both cases
we list the zeros by increasing imaginary part.
We consider the normalized gaps between zeros of~$\zeta$ and
the normalized distance of $\rho_j'$ to the right of the
critical line, given by
\begin{align}\label{def:lambdas}
\lambda_j=\mathstrut &(\gamma_{j+1}-\gamma_j)\log\gamma_j \cr
\lambda_j'=\mathstrut & (\beta_j'-\tfrac12) \log\gamma_j'.
\end{align}
We are interested in how small the normalized gaps can be,
and how small the normalized distance to the critical line
can be, so we set
\begin{align}
\lambda=\mathstrut &\liminf_{j\to\infty} \lambda_j\\
\lambda'=\mathstrut &\liminf_{j\to\infty}\lambda_j' .
\end{align}
We also consider the cumulative densities of $\lambda_j$ and $\lambda_j'$,
given by
\begin{align}
m(\nu) =\mathstrut & \liminf_{J\to\infty}
\frac{1}{J} \, \#\{j\le J\ :\ \lambda_j \le \nu\}\cr
m'(\nu) =\mathstrut & \liminf_{J\to\infty} \frac{1}{J}\,
\#\{j\le J\ :\ \lambda_j' \le \nu\}.
\end{align}

Soundararajan's~\cite{S} Conjecture~B states that $\lambda=0$ if and only if
$\lambda'=0$.  
This amounts to conjecturing that zeros of $\zeta'(s)$ close to the
$\tfrac12$-line can only arise from a pair of closely spaced
zeros of~$\zeta(s)$.
Zhang~\cite{Z} showed that (on RH) $\lambda=0$ implies
$\lambda'=0$.
Thus, Soundararajan's conjecture is almost certainly true
because $\lambda=0$ follows from standard conjectures on
the zeros of the zeta-function, based on random matrix theory.

However, the second author\cite{K} showed that $\lambda=0$ and $\lambda'=0$
are not logically equivalent.  Specifically, Ki\cite{K} proved
\begin{theorem}\label{thm:ki} (Haseo Ki \cite{K})  Assuming RH,
$\lambda' >0$ is equivalent to 
\begin{equation}\label{eqn:zetacondition}
M(\gamma_j):= \sum_{0<|\gamma_j-\gamma_n|<1} \frac{1}{\gamma_j-\gamma_n}
=O(\log \gamma_j) .
\end{equation}
\end{theorem}
 
Note that the theorem implies Zhang's result (that $\lambda=0$
implies $\lambda'=0$), because if $\lambda=0$ then for some $j$
the sum
in \eqref{eqn:zetacondition} will be large because an individual
term in the sum is large.  But that is not the only way for 
$M(\gamma_j)$
to be large.  It is possible that there could be an imbalance
in the distribution of zeros, such as a very large gap between
neighboring zeros, which makes the sum 
large because many small terms have the same sign.  

For example, suppose there were consecutive zeros of the
zeta function with a gap
of size 1, followed by $c \log T$
zeros equally spaced (this cannot happen, but we are illustrating
a point).  Then $M(\gamma)$
would be $\gg \log T \log\log T$. 
That possibility is the reason attempts to prove
$\lambda'=0$ implies $\lambda=0$ have been unsuccessful.
For example, Garaev and Y{\i}ld{\i}r{\i}m~\cite{GY}
required the stronger assumption
$\lambda_J'(\log\log \gamma_J')^2=o(1)$ in order to conclude
$\lambda_J=o(1)$.

The discussion in the previous paragraph shows that,
without detailed knowledge of the distribution of zero spacings,
one requires $ M(\gamma)\ge C \log T \log\log T$ for
any $C>0$ in order to conclude $\lambda =0$.  
It is possible that this could be improved by proving
results about the rigidity of the spacing between zeros of
the zeta function.  Random matrix theory could give a clue
about the limits of this approach.  This would involve finding
the expected maximum of the random matrix analogue of the
sum
\begin{equation}\label{eqn:zetairregular}
\sum_{\frac{1}{\log \gamma_j}<|\gamma_j-\gamma_n|<1}
\frac{1}{\gamma_j-\gamma_n}.
\end{equation}
Unfortunately, the necessary random matrix
calculation may be quite difficult
because a lower bound on $|\gamma_j-\gamma_n|$ requires the exclusion
of a varying number of intervening zeros, so the combinatorics
of the random matrix calculation may be intricate.

In this paper we consider not $\lambda$ and $\lambda'$,
but the density functions $m(\nu)$ and $m'(\nu)$.
In the next section we illustrate this with the example described above,
and then we state our main result.

\subsection{Examples with equally spaced zeros}\label{sec:pictures}

We illustrate Theorem~\ref{thm:ki} with examples 
which can help build intuition for why $\lambda'=0$
does not imply $\lambda=0$. 

Our example involves degree $N$ polynomials with all zeros on the
unit circle.  In other words, characteristic polynomials of
matrices in the unitary group~$U(N)$.
In these examples.
$\lambda>0$ but $\lambda'=0$, where $\lambda$ and $\lambda'$
refer respectively to the large $N$ limits of the normalized gap between
zeros, and the rescaled distance between zeros of the
derivative and the unit circle.
This is the random matrix analogue of $\lambda$ and $\lambda'$
for the zeta function.

Figure~\ref{fig:1a} illustrates the case of 16 zeros in the
interval$\{e^{i\theta}\ :\ 0\le \theta\le \pi/2\}$.
The plot on the left shows the zeros of the polynomial and
its derivative.  The figure on the right is the same plot
``unrolled'':  the horizontal axis is the argument,
and the vertical axis is the distance from the unit circle,
rescaled by a constant factor.

\begin{figure}[htp]
\begin{center}
\scalebox{0.7}[0.7]{\includegraphics{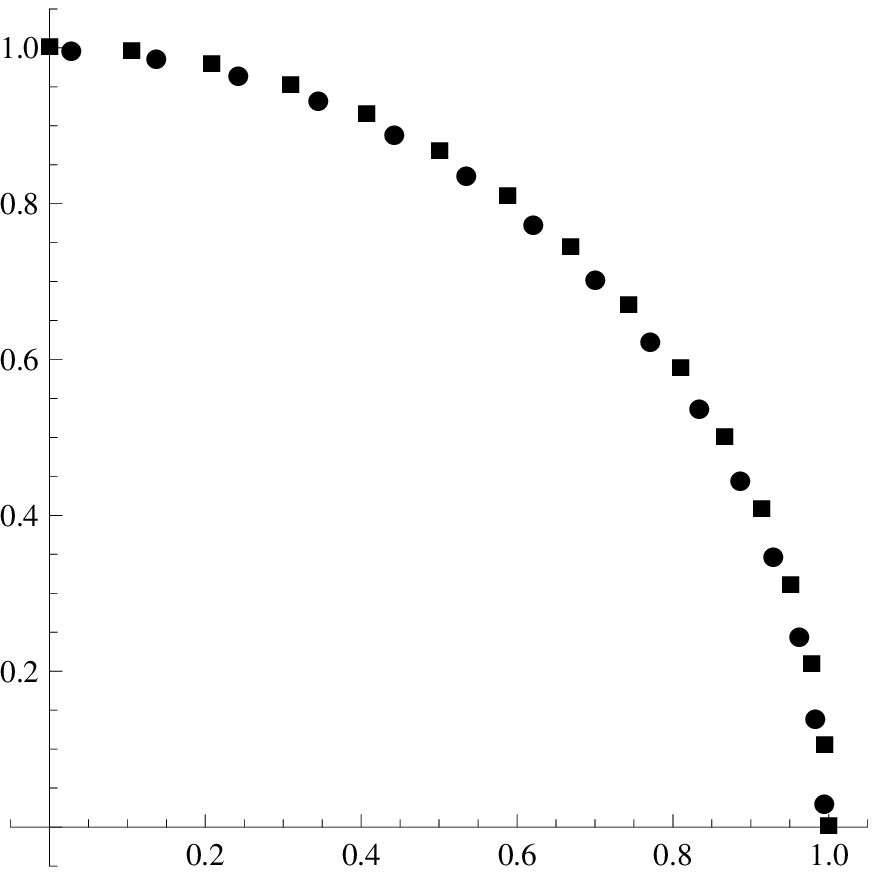}}
\hskip 0.5in
\scalebox{0.7}[0.7]{\includegraphics{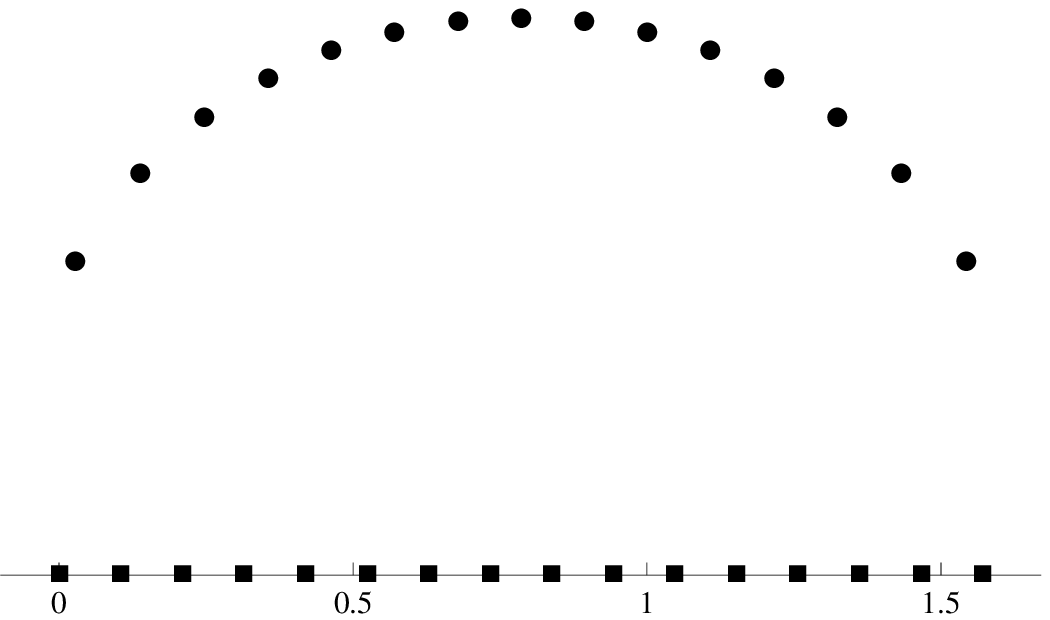}}
\caption{\sf
On the left, the zeros and the zeros of the derivative of a
degree 16 polynomial having all zeros in $\frac14$ of the
unit circle.  On the right, the image of those zeros 
under the mapping $r e^{i \theta} \mapsto (\theta, 2 \pi \cdot 16(1-r))$.
Zeros of the function are shown as small squares, and zeros of
the derivative as small dots.
} \label{fig:1a}
\end{center}
\end{figure}

Figure~\ref{fig:1b1c} is the analogue of the plot on the
right side of Figure~\ref{fig:1a}, for 101 zeros and 501 zeros.
Note that in these examples $\lambda\sim \pi/2$.

\begin{figure}[htp]
\begin{center}
\scalebox{0.7}[0.7]{\includegraphics{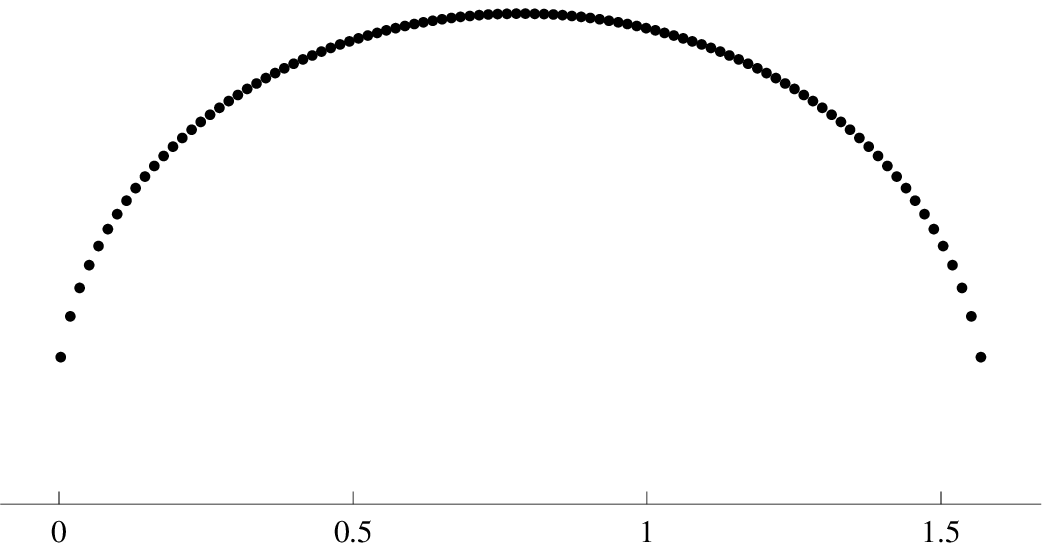}}
\hskip 0.5in
\scalebox{0.7}[0.7]{\includegraphics{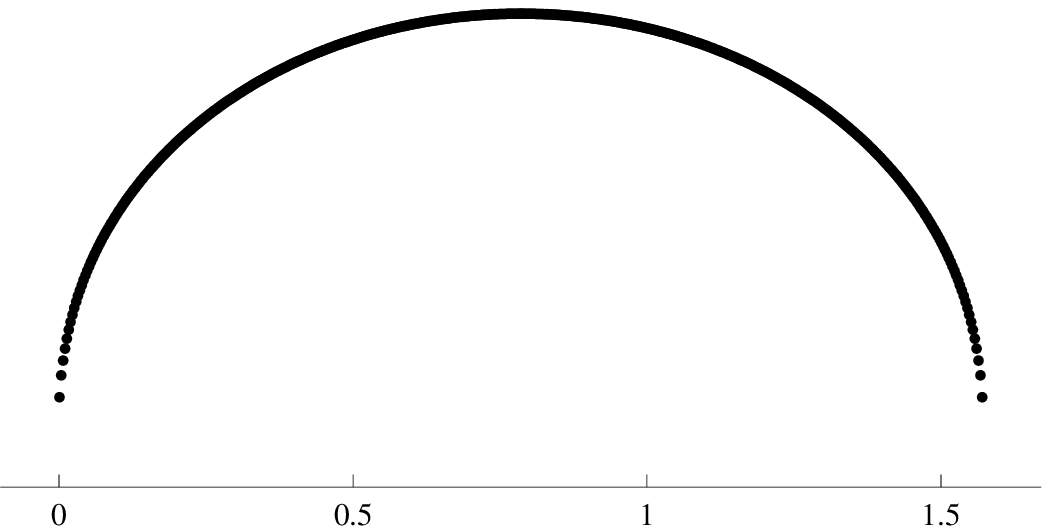}}
\caption{\sf
Unrolled and rescaled zeros of the derivative
of a polynomial with  zeros equally spaced along the arc
$\{e^{i\theta}\ :\ 0\le \theta\le \pi/2\}$.
The polynomial has degree
101 (left) and 501 (right).
} \label{fig:1b1c}
\end{center}
\end{figure}
In Figure~\ref{fig:1b1c} the vertical scales are stretched
by a factor of
$2 \pi N (1-r)$ where $N=101$ and $501$, respectively.

Figures~\ref{fig:1a} and \ref{fig:1b1c} illustrate that, with this
unrolling and rescaling, the zeros of the derivative approach
a circle.   
We see that even though $\lambda>0$ we have $\lambda'=0$,
but furthermore,
since the zeros lie on a (rescaled)
circle, we have $m'(\nu)\gg \nu^2$ as $\nu\to 0$.  Thus,
we can have $m'(\nu)>0$ for all $\nu>0$, yet $m(\nu)=0$ for
$\nu$ sufficiently small.

We believe that the above example is the limit of this
behavior, and we make the following conjecture, which we
view as a refinement of Soundararajan's conjecture.

\begin{conjecture}\label{conj:conjecture1} If $m'(\nu) \gg \nu^\alpha$ for some $\alpha<2$,
then $m(\nu)>0$ for all $\nu>0$.
\end{conjecture}

We intend this as a general conjecture, applying
to the Riemann zeta function but also to other
cases such as a
sequence of polynomials with all zeros on the unit circle.

For applications to lower bounds of class numbers~\cite{M,CI}
one does not actually need $m(\nu)>0$ for $\nu < \pi $; it is sufficient
to show that a relatively small number of gaps between zeros
of the zeta function are small.  Our main result, Theorem~\ref{t:la},
obtains such bounds from estimates on the zeros of the derivative
of the zeta function.

Denote $ \log_{(2)}t=\log\log t $.

\begin{theorem}\label{t:la}
Assume RH. Suppose that for all $\nu>0$,
\begin{equation}
\#\{0<\gamma'<T:\left(\beta'-\tfrac{1}{2}\right)\log\gamma'\le\nu\}\geqslant
e^{-C(\nu)}T\log T\qquad(\nu>0,\,T\to\infty),
\end{equation}
as $T\to \infty$,
where $C(\nu)>0$ for $\nu>0$ with
$\lim_{\nu\to0^+}\sqrt{\nu}C(\nu)=0$ and
$\lim_{\nu\to0^+}C(\nu)=\infty$.
Then
\begin{equation}\label{eqn:mainresult}
\liminf_{T\to\infty}\frac{\#\{\gamma_n\le
T:(\gamma_{n+1}-\gamma_n)\log\gamma_n\le\nu\}}{T\log T/\logg
T}>0
\end{equation}
for all $\nu>0$.
\end{theorem}

The conclusion of the theorem is weaker than $m(\nu)>0$
for $\nu>0$, but only by a factor of $\logg T$.  Thus, it is more
than sufficient to apply the results of Conrey and Iwaniec~\cite{CI}.
In particular, Theorem~\ref{t:la} shows that it is possible to
obtain lower bounds for class numbers of imaginary quadratic fields
from knowledge of the density of zeros of the derivative
of the Riemann zeta function.

There is an apparent discrepancy between Conjecture~\ref{conj:conjecture1}
and Theorem~\ref{t:la} which we wish to clarify.  In Theorem~\ref{t:la}
we allow exponential decrease of $m'(\nu)$ as $\nu\to 0$.  While
the conclusion of the theorem is weaker than $m(\nu)>0$ by a
factor of $\logg T$, it may seem curious that the condition in
Conjecture~\ref{conj:conjecture1} requires  $m'(\nu)$ 
to be relatively large as $\nu\to 0$.  Indeed, the examples
in Section~\ref{sec:pictures} show that the condition in
Conjecture~\ref{conj:conjecture1} cannot be improved for
general functions.

The reason for the apparent inconsistency is that, as described
in Section~\ref{sec:boundMgammac}, our method relies on a bound
on the moments of the logarithmic derivative.  For the Riemann
zeta function one expects
\begin{equation}\label{eqn:zpzbound}
\int_{T}^{2T} \left|
  \frac{\zeta'}{\zeta}\left(\frac12 + \frac{1}{\log T} + i t\right)
\right|^{2k} dt \ll_k T\log^{2k}T.
\end{equation}
The bound \eqref{eqn:zpzbound} should follow by the method
of Selberg~\cite{Sel},
although we give a conditional proof that allows us to
explicitly determine the implied constant.
Such a bound, for one fixed~$k$,
would establish a weaker version of Theorem~\ref{t:la} that 
required $m'(\nu)\gg \nu^{-2 k}$.  However, more general functions
like the polynomials in Section~\ref{sec:pictures}
do not satisfy an analogous bound to \eqref{eqn:zpzbound}.  In fact, they are
very large on the unit circle  and do not satisfy the analogue
of the Lindel\"of hypothesis.
Conjecture~\ref{conj:conjecture1} is intended to cover those
more general cases, while stronger statements should be true
for the zeta function.

It is interesting to speculate on the precise nature of the
function $m'(\nu)$ for the Riemann zeta function.
Due\~nez~{\it et. al.}~\cite{Due}
give a detailed analysis of the relationship between small gaps
between zeros of the zeta function (and analogously for zeros of
the characteristic polynomial of a random unitary matrix) and
the zeros of the derivative which arise from the small gaps.
For the case of the Riemann zeta function they indicate that
the random matrix conjectures for the zeros of the zeta function
should imply
\begin{equation}
m_\zeta'(\nu) \sim \frac{8}{9\pi} \nu^{\frac{3}{2}},
\end{equation}
as conjectured by Mezzadri~\cite{Mez03}.
That calculation is based on a more general result which suggests
that if $m(\nu)\sim \kappa \nu^\beta$ then $m'(\nu) \sim \kappa' \nu^{\beta/2}$
where 
\begin{equation}
\kappa'=2\pi \frac{\kappa}{\beta} \left(\frac{2}{\pi}\right)^\beta.
\end{equation}
The factor of $2\pi$ comes from a different normalization used in~\cite{Due}
and here we work with the cumulative distribution functions
$m$ and $m'$, while in~\cite{Due} they use density functions.
That derivation assumed that zeros of $\zeta'$
close to the $\tfrac12$-line only arise from closely spaced
zeros of the zeta-function. The discussion above shows that,
without further knowledge of the zeros, this is not a valid
assumption.  But, as indicated in our Conjecture~\ref{conj:conjecture1},
if $\beta<4$ then we believe that the almost all zeros close
to the $\tfrac12$-line do arise in such a manner.  The 
random matrix prediction for the neighbor spacing
of zeros of the zeta-function has $\kappa=\pi/6$
and $\beta=3$, which is covered by Conjecture~\ref{conj:conjecture1}.
So our results support the analysis of Due\~nez~{\it et. al.}~\cite{Due}.

The remainder of this paper is devoted to the proof of Theorem~\ref{t:la}.

\section{Proof of Theorem~\ref{t:la}}

Theorem~\ref{t:la} says that sufficiently many zeros of $\zeta'$ close to the
$\frac12$-line
can only arise from closely spaced zeros of the zeta-function.
If $\rho'=\beta'+i\gamma'$ is a zero of $\zeta'$, then we denote
by $\rho_c=\frac12+i\gamma_c$ the zero of the zeta-function which is
closest to~$\rho'$. 
Thus, we must show that if there are many
$\beta'$ very close to $\tfrac12$, then often there is another
zero of the zeta-function close to~$\gamma_c$.

Our approach involves a study of the quantity
\begin{equation}
M_{\gamma_c}=\sum_{0<|\gamma-\gamma_c|\le
X(\gamma_c)} 
\frac{1}{\gamma-\gamma'},
\end{equation}
where the range in the sum, $X(\gamma_c)$, turns out to be a limiting
factor in our method.  By analogy to a similar quantity studied in~\cite{K},
we expect that $M_{\gamma_c}$ should be large if and only if
$\beta'-\frac12$ is small.  And just like in~\cite{K},
there are two ways that $M_{\gamma_c}$ can be large.  There could be
an individual term which is large.  That would happen if $\gamma'$ was near
two $\gamma$s that are very close together.  Or there could be a large
imbalance in the the distribution of the $\gamma$s, 
for example if there was an unusually large
gap between $\gamma_c$ and one of the adjacent zeros.  We must show that the
second possibility cannot occur too often.
This is accomplished by showing that an imbalance in the distribution
of zeros causes the zeta function to be large, and bounds on moments of the
zeta function show that this cannot happen too often. 

The proof involves two steps.  Assume the zeros of the
zeta function rarely get close together.  First we show that if
$\beta'-\frac12$ is small then $M_{\gamma_c}$ is large.
Second,  we show that if  $M_{\gamma_c}$ is large then
usually $\frac{\zeta'}{\zeta}(s)$
is large near $\tfrac12+i\gamma'$, subject to our assumption
that the zeros of the zeta function rarely get close together.
Standard bounds for the
moments of~$\frac{\zeta'}{\zeta}(\sigma+i t)$ let us conclude
that $\beta'-\frac12$ cannot be small too often, which is
what we wanted to prove.

The relationship between $M_{\gamma_c}$ and  $\zeta'/\zeta$
relies on an estimate for $\zeta'/\zeta$ in terms of a short sum
over zeros.  Suppose we have
\begin{equation}\label{eqn:logderiv1}
\frac{\zeta'}{\zeta}(s) = \sum_{|\gamma-t|< X(T)} \frac{1}{s-\rho} + O(\log T).
\end{equation}
On RH, with $X(t)=1/\logg T$ the above holds for all~$t$~\cite{Ti}.
Using this, instead
of our~\eqref{e:logdea} below, leads to a weaker version of Theorem~\ref{t:la},
where the $\logg T$ in the denominator of \eqref{eqn:mainresult} is replaced by~$\log T$.

We prove the following strengthening of~\eqref{eqn:logderiv1},
but only near almost all~$\gamma$.

\begin{proposition}\label{p:logdea}
Assume RH. Let $m_0$ be a positive integer.
If $C^*>1$ is sufficiently large, then the number of $\gamma_n<T$ such
that
\begin{equation}\label{e:logdea}
\frac{\zeta'}{\zeta}(s)=
\sum_{|\gamma-t|\le
\frac{C^*\logg\gamma}{\log\gamma}}\frac{1}{s-\rho}
+
O(\log\gamma_n)
\end{equation}
for $s=1/2+1/\log\gamma_n+it$ with $t\geqslant10$ and
$|\gamma_n-t|\le A/\log\gamma_n$ is
\begin{equation}
\frac{T}{2\pi}\log T+O\left(\frac{T}{(\log T)^{m_0}}\right)
\end{equation}
as $T\to\infty$.
\end{proposition}

The proof of~Proposition~\ref{p:logdea} is in Section~\ref{sec:logdea}.

\subsection{Restricting to zeros with special properties}
We begin the proof of Theorem~\ref{t:la}.
The lemmas in this section show that, in the context of the proof
of Theorem~\ref{t:la}, we only have to deal with zeros that are
well spaced.

Suppose, for the purposes of contradiction, that there exists
$\epsilon>0$ so that
\begin{equation}\label{eqn:nosmallgaps}
\liminf_{T\to\infty}\frac{\#\{\gamma_n\le
T:\gamma_{n+1}-\gamma_n\le\epsilon/\log\gamma_n\}}{T\log
T/\logg T}=0.
\end{equation}
Then, we can find a sequence $\langle T_l\rangle$ such that $T_1$ is
sufficiently large, $T_l\to\infty$ and
\begin{equation}
\#\{\gamma_n\le
T_l:\gamma_{n+1}-\gamma_n\le\epsilon/\log\gamma_n\}=o\left(T_l\log
T_l/\logg T_l\right)
\end{equation}
as $l\to\infty$. We set
\begin{equation}
T=T_l.  
\end{equation}

The following lemma shows that we can restrict our attention to those
zeros whose immediate neighbors are well spaced.

\begin{lemma}\label{lem:density}  Let $K=4C^*\left[\logg T\right]$.  
Under assumption~\eqref{eqn:nosmallgaps} we have
\begin{equation}\label{e:density}
\# \{\gamma_n<T:0<|m|\le
K,\,|\gamma_{n+m}-\gamma_{n+m-1}|\geqslant\frac{\epsilon}{2\log\gamma_n}\}=\frac{T}{2\pi}\log
T(1+o(1)).
\end{equation}
\end{lemma}

\begin{proof}
For each $m=\pm1,\pm2,\ldots$, let
\begin{equation}
A_m=\{\gamma_n<T:|\gamma_{n+m}-\gamma_{n+m-1}|\geqslant\frac{\epsilon}{2\log\gamma_n}\}.
\end{equation}
Here, we exclude the case $n+m\le1$.
By assumption~\eqref{eqn:nosmallgaps} have
\begin{equation}
\#(A_m)=\frac{T}{2\pi}\log T+o\left(\frac{T\log T}{\logg T}\right)
\end{equation}
for $0<|m|\le\log T$.
We see that
\begin{equation}
\aligned \#\left(\bigcap_{0<|m|\le
K}A_m\right)=&\sum_{0<|m|\le
K}\#(A_m)-\sum_{\substack{-K\le
m<K\\ m\not=0}}\#\left(A_m\cup\bigcap_{\substack{m<l\le K\\
l\not=0}}A_l\right)\\
\geqslant&2K\frac{T}{2\pi}\log T+o\left(\frac{KT\log T}{\logg
T}\right)-(2K-1)\frac{T}{2\pi}\log
T+O(K\log T)\\
=&\frac{T}{2\pi}\log T+o(T\log T).
\endaligned
\end{equation}

\end{proof}

The next Proposition shows that we can restrict to intervals
where the number of
zeros is close to its average.
Fix $C^*>1$, let $l_1$ and $l_2$ be integers, and for
$\tfrac12+i\gamma$ a zero of the zeta function set
\begin{equation}
N(\gamma,l_1,l_2)=
N\left(\gamma+\frac{l_2C^*\logg\gamma}{\log\gamma}\right)
-N\left(\gamma+\frac{l_1C^*\logg\gamma}{\log\gamma}\right)
-\frac{(l_2-l_1)C^*\logg\gamma}{2\pi}
\end{equation}
Using an argument in \cite{Iv}, we get the following.
\begin{proposition}\label{p:dnumber} Let $m_0>0$. There exists $C>0$
such that the
number of $\gamma_n<T$ with
\begin{equation}
N(\gamma_n,l_1,l_2)\le C\logg T
\end{equation}
is
\begin{equation}
\frac{T}{2\pi}\log T+O\left(\frac{T}{(\log T)^{m_0}}\right)\qquad(T\to\infty),
\end{equation}
provided that
$|l_1|,|l_2|\le\log T/(C^*\logg T)$ and
$0<l_2-l_1\le2\log T/(C^*\logg T)$.
\end{proposition}

The proof of Proposition~\ref{p:dnumber} is in Section~\ref{sec:dnumber}.

\subsection{Lower bound for $M_{\gamma_c}$}
Let $\beta'+i\gamma'$ be a zero of $\zeta'$, and (assuming RH)
let $\tfrac12+i\gamma_c$ be the zero of the zeta function
which is closest to~$i\gamma'$.  If there are two closest
zeros, choose the one nearer to the origin.
We will use the above lemmas to give a lower bound for  $M_{\gamma_c}$,
assuming $\beta'-\tfrac12$ is small.  

Let $Z(T)$ be the set of $\gamma_c < T$ which satisfy the following
three conditions:
\begin{align}
\gamma_c\in\mathstrut & \{\gamma_n<T :
	0<|m|\le K,\,
	|\gamma_{n+m}-\gamma_{n+m-1}|\geqslant
			\frac{\epsilon}{2\log\gamma_n}\};
\label{eqn:gammacset}\\
N(\gamma_c,l_1,l_2)\le\mathstrut & 
C\logg T\qquad
	\left(-\frac{\log T}{C^*\logg T}\le l_1<l_2\le\frac{\log T}{C^*\logg T}\right);
\label{eqn:gammacN}\\
\frac{\zeta'}{\zeta}(s)=\mathstrut & 
   \sum_{|\gamma-t|\le\frac{\logg\gamma}{\log\gamma}}
    \frac{1}{s-\rho}
\label{eqn:gammaczpz}
+
O(\log \gamma_c),
\end{align}
where $s=1/2+1/\log\gamma_c+it$ and $|\gamma_c-t|\le
A/\log\gamma_c$.  By the lemmas in the previous section,
as $T\to\infty$ the set
$Z(T)$ contains $\sim \frac{1}{2\pi}T\log T$ elements.
For the remainder of the proof we will assume~$\gamma_c\in Z(T)$.

\medskip
 
Recall Titchmarsh~\cite{Ti}, Theorem 9.6(A):
\begin{equation}\label{eqn:zpzeasy}
\frac{\zeta'}{\zeta}(s)=-\frac{1}{2}\log
t+O(1)+\sum_{\rho}\left(\frac{1}{s-\rho}-\frac{1}{\rho}\right),
\end{equation}
uniformly for $t\geqslant10$ and
$-1\le\sigma\le2$.
Let $\beta'+i\gamma'$ be a zero of $\zeta'(s)$ where $0<\gamma'<T$
is sufficiently large. Taking the real part \eqref{eqn:zpzeasy}
we have
\begin{equation}\label{e:gprime}
\frac{1}{2}\log\gamma'+O(1)
=
\frac{\beta'-\frac{1}{2}}
	{\left(\beta'-\frac{1}{2}\right)^2+(\gamma'-\gamma_c)^2}
+
\sum_{\gamma\not=\gamma_c}
\frac{\beta'-\frac{1}{2}}
	{\left(\beta'-\frac{1}{2}\right)^2+(\gamma'-\gamma)^2}.
\end{equation}

There are three cases to consider.

\noindent {\bf Case 1.} $\beta'-1/2>|\gamma'-\gamma_c|$.\medskip

Then, by (\ref{e:gprime}), we get
\begin{equation}
\frac{1}{2}\log\gamma'\geqslant\frac{1}{2\left(\beta'-\frac{1}{2}\right)}.
\end{equation}
Thus, we have $\beta'-1/2 \gg 1/\log\gamma'$.\medskip

\noindent {\bf Case 2.} $\beta'-1/2\le|\gamma'-\gamma_c|$ and
$|\gamma'-\gamma_c|>\delta(\epsilon)/\log\gamma'$, where
$\delta(\epsilon)=8/\epsilon^2$. \medskip

By \eqref{e:gprime}, \eqref{eqn:gammacset}, and \eqref{eqn:gammacN}, we have
\begin{equation}
\aligned
\frac{1}{2}\log\gamma'\ll&\left(\beta'-\frac{1}{2}\right)\log^2\gamma'+\sum_{m=1}^{\infty}
\frac{\beta'-\frac{1}{2}}{\left(\frac{m\epsilon}{\log\gamma'}\right)^2}+\sum_{m=0}^{\infty}
\frac{\left(\beta'-\frac{1}{2}\right)\logg\gamma'}
{\left(\frac{\logg\gamma'}{\log\gamma'}\right)^2+\left(\frac{m\logg\gamma'}{\log\gamma'}\right)^2}\\
\ll&\left(\beta'-\frac{1}{2}\right)\log^2\gamma'
\endaligned
\end{equation}
and so again we have
\begin{equation}
\beta'-\frac{1}{2}\gg\frac{1}{\log\gamma'}.
\end{equation}
Here the implied constants depend only on $\epsilon$.
\medskip

\noindent {\bf Case 3.} $\beta'-1/2\le|\gamma'-\gamma_c|$ and
$|\gamma'-\gamma_c|\le\delta(\epsilon)/\log\gamma'$. \medskip

Using \eqref{e:gprime}, \eqref{eqn:gammacset}, and \eqref{eqn:gammacN},
as in Case 2, we get
\begin{equation}
\frac{1}{2}\log\gamma'\geqslant\frac{\beta'-\frac{1}{2}}{2(\gamma'-\gamma_c)^2}
\end{equation}
\begin{equation}
\frac{1}{2}\log\gamma'\ll\frac{\beta'-\frac{1}{2}}{(\gamma'-\gamma_c)^2}+\left(\beta'-\frac{1}{2}\right)\log^2\gamma'\\
\ll\frac{\beta'-\frac{1}{2}}{(\gamma'-\gamma_c)^2}.
\end{equation}
Thus we have
\begin{equation}\label{eqn:twosided}
(\gamma'-\gamma_c)^2\log\gamma'\ll
\beta'-\frac{1}{2}\ll(\gamma'-\gamma_c)^2\log\gamma'.
\end{equation}
Here the implied constants depend only on $\epsilon$.
By~\eqref{eqn:twosided} and the conditions of Case~3 we have
\begin{equation}\label{eqn:nosum}
\frac{\gamma_c-\gamma'}{\left(\beta'-\frac{1}{2}\right)^2+(\gamma'-\gamma_c)^2}
	-\frac{1}{\gamma_c-\gamma'}=O(\log\gamma').
\end{equation}

Now take the imaginary part of~\eqref{eqn:gammaczpz} to get
\begin{equation}\label{eqn:sum1}
\sum_{0<|\gamma-\gamma_c| \le \frac{C^*\logg\gamma_c}{\log\gamma_c}}
\frac{\gamma-\gamma'}{\left(\beta'-\frac{1}{2}\right)^2+(\gamma'-\gamma)^2}
+\frac{\gamma_c-\gamma'}
	{\left(\beta'-\frac{1}{2}\right)^2+(\gamma'-\gamma_c)^2}
= O\left(\log\gamma'\right).
\end{equation}

Finally, by \eqref{eqn:gammacset} we have
\begin{align}\label{eqn:Mcsum}
\sum_{0<|\gamma-\gamma_c|\le
\frac{C^*\logg\gamma_c}{\log\gamma_c}}
	\frac{\gamma-\gamma'}{\left(\beta'-\frac{1}{2}\right)^2
	+(\gamma'-\gamma)^2}-M_{\gamma_c}
=\mathstrut&
\sum_{k=1}^{\infty}\frac{\left(\beta'-\frac{1}{2}\right)^2}{\left(\frac{\epsilon k}{\log\gamma'}\right)^3}\\
=\mathstrut&O(\log\gamma'),
\end{align}
where
\begin{equation}
M_{\gamma_c}=\sum_{0<|\gamma-\gamma_c|\le
\frac{C^*\logg\gamma_c}{\log\gamma_c}}\frac{1}{\gamma-\gamma'}.
\end{equation}
By combining \eqref{eqn:nosum}, \eqref{eqn:sum1}, \eqref{eqn:Mcsum},
and \eqref{eqn:twosided}, we have
\begin{equation}
O(\log\gamma')=M_{\gamma_c}+\frac{1}{\gamma_c-\gamma'}=M_{\gamma_c}
+ A_{\gamma_c} \sqrt{\frac{\log\gamma'}{\beta'-\frac{1}{2}}},
\end{equation}
where $1\ll A_{\gamma_c}\ll 1 $, with the implied constants
depending only on~$\epsilon$.

\medskip

Let $\nu$ be a positive number. Suppose that
\begin{equation}
\left(\beta'-\frac{1}{2}\right)\log\gamma'\le\nu.
\end{equation}
Then, for sufficiently small $\nu$, we see that
only Case 3 is possible for sufficiently large $\gamma'$, namely we
have
\begin{equation}
M_{\gamma_c}+
	 A_{\gamma_c} \sqrt{\frac{\log\gamma'}{\beta'-\frac{1}{2}}}
=O(\log \gamma').
\end{equation}
By this, the assumption in Theorem~\ref{t:la}, and the fact that
$\#Z(T)\sim \frac{1}{2\pi} T\log T$, we have
\begin{equation*}
e^{-C(\nu)}\le\frac{1}{\frac{T}{2\pi}\log
T}\#\{0<\gamma'<T: \gamma_c\in Z(T) \text{ and } |M_{\gamma_c}|\gg \frac{\log\gamma'}{\sqrt{\nu}}(1+O(\sqrt{\nu}))\}.
\end{equation*}
By the last inequality we have
\begin{equation}\label{e:nui}
\frac{e^{-C(\nu)}\log^{2k}T}{\nu^{k}}(1+O(\sqrt{\nu}))^{2k}\frac{T}{2\pi}\log
T\ll\sum_{\substack{\frac{T}{\log T}\le\gamma'\le T\\
	\left(\beta'-\frac{1}{2}\right)\log\gamma'\le\nu\\
	\gamma_c\in Z(T)}}
|M_{\gamma_c}|^{2k}.
\end{equation}

In the next section we describe upper bounds for the moments of $M_{\gamma_c}$.
This will contradict~\eqref{e:nui} and complete the proof of~Theorem~\ref{t:la}.

\subsection{Bounding the moments of $M_{\gamma_c}$}\label{sec:boundMgammac}

We obtain an upper bound on $M_{\gamma_c}$ from a bound on moments
of the logarithmic derivative of the zeta function.  This makes use of
that fact that, assuming
the zeros of the zeta function do not get close together,
the logarithmic derivative can be approximated either by a short
sum over zeros, or by a short Dirichlet series.

\begin{lemma}\label{l:moment} Assume RH.
Let $\gamma'<T$ such that
$|\gamma'-\gamma_c|\le\delta(\epsilon)/\log\gamma'$
and assume \eqref{eqn:gammacset} -- \eqref{eqn:gammaczpz}. Then
\begin{equation}
\frac{M_{\gamma_c}}{i}+\frac{\zeta'}{\zeta}\left(\frac{1}{2}+\frac{1}{\log
T}+it\right)=O_\epsilon (\log T),
\end{equation}
for $|t-\gamma'|\le A/\log\gamma'$.
\end{lemma}

\begin{proof} By the assumptions
we have
\begin{align}
\frac{M_{\gamma_c}}{i}+\frac{\zeta'}{\zeta}\left(\frac{1}{2}+\frac{1}{\log
T}+it\right)
=\mathstrut&\frac{M_{\gamma_c}}{i}+\sum_{0<|\gamma-\gamma_c|\le
\frac{C^*\logg\gamma_c}{\log\gamma_c}}\frac{1}{\frac{1}{\log T}+i(t-\gamma)}
	+O(\log T)\\
=\mathstrut&\sum_{0<|\gamma-\gamma_c|\le
\frac{C^*\logg\gamma_c}{\log\gamma_c}}
	\frac{(\frac{1}{\log T}+i (t-\gamma))}
	{(\gamma-\gamma')(\frac{1}{\log T}+i(t-\gamma))}
+O(\log T)\\
\ll\mathstrut &\sum_{m=1}^{\infty}
     \frac{\frac{1+\delta(\epsilon)}{\log\gamma'}}
		{\left(\frac{m\epsilon}{\log\gamma'}\right)^2}+O(\log T)\\
=\mathstrut &O_\epsilon (\log T).
\end{align}
\end{proof}

\begin{lemma}\label{lem:zpzapprox} Assume RH and
\eqref{eqn:gammacset} -- \eqref{eqn:gammacN}.
Let $s=\frac12+\frac{1}{\log T}+i t$ with $|t|\le T$, and 
let $x=T^{1/100 k}$.  
Then if
$|\gamma'-\gamma_c|\le\delta(\epsilon)/\log\gamma'$
and $|t-\gamma'|\le \epsilon /\log\gamma'$, we have
\begin{equation}
\frac{\zeta'}{\zeta}(s)=-\sum_{n<x^2}\frac{\Lambda_x(n)}{n^s}+O_\epsilon(k\log T),
\end{equation}
where
\begin{equation}
\Lambda_x(n)=
\begin{cases}
\Lambda(n) & 1\le n\le x \cr
\Lambda(n)\frac{\log(\frac{x^2}{n})}{\log x} & x\le n\le x^2
\end{cases} .
\end{equation}
\end{lemma}

\begin{proof}
By \cite{Ti}, Theorem 14.20,
\begin{align}
\frac{\zeta'}{\zeta}(s)=\mathstrut &-\sum_{n<x^2}\frac{\Lambda_x(n)}{n^s}+\frac{x^{2(1-s)}-x^{1-s}}{(1-s)^2\log
x}\cr
&+ \frac{1}{\log
x}\sum_{q=1}^{\infty}\frac{x^{-2q-s}-x^{-2(2q+s)}}{(2q+s)^2}+\frac{1}{\log
x}\sum_{\rho}\frac{x^{\rho-s}-x^{2(\rho-s)}}{(s-\rho)^2}.
\end{align}
The assumptions on the zero spacings give the 
claimed bound on the terms involving the zeros.
\end{proof}

\begin{lemma}\label{lem:sound} (Soundararajan, Lemma 3 of \cite{So1}) Let $T$ be large, and let $2\le x\le T$. Let $k$ be a natural number such
that $x^k\le T/\log T$. For any complex numbers $a(p)$ we have
$$
\int_T^{2T}\left|\sum_{p\le x}\frac{a(p)}{p^{\frac{1}{2}+it}}\right|^{2k}dt\ll
k!\, T \left(\sum_{p\le x} \frac{|a(p)|^2}{p}\right)^{k} ,
$$
where the sum is over the primes.
\end{lemma}

We assemble the above lemmas to bound the moments of~$M_{\gamma_c}$.

By Lemma~\ref{l:moment} and Lemma~\ref{lem:zpzapprox}, with
$A=A_\epsilon$ a constant depending only on $\epsilon$, 
which may be different in each inequality, we have
\begin{align}\label{eqn:Mgammacbound}
|M_{\gamma_c}|^{2k} \ll \mathstrut & A^{2k} \log^{2k}T +  2^{2k} \left|
        \frac{\zeta'}{\zeta}\left(\frac12 + \frac{1}{\log T} + i t\right)
   \right|^{2k} \cr
\ll \mathstrut & A^{2k} k^{2k} \log^{2k}T + 
2^{2k}
\left|\sum_{n<x^2}
   \frac{\Lambda_x(n)}{n^{\frac12 + \frac{1}{\log T} + i t}}\right|^{2k}\cr
\ll \mathstrut & A^{2k} k^{2k} \log^{2k}T + 
2^{2k}
\left|\sum_{p<x^2}
   \frac{\Lambda_x(p)}{p^{\frac12 + \frac{1}{\log T} + i t}}\right|^{2k},
\end{align}
where $x=T^{1/100k}$,
for $|t-\gamma'|\le \delta(\epsilon)/\log\gamma'$,
provided $\gamma_c$ satisfies \eqref{eqn:gammacset} -- \eqref{eqn:gammaczpz}.
That is, provided $\gamma_c\in Z(T)$.

Integrating inequality~\eqref{eqn:Mgammacbound} over the set
\begin{equation}
\left\{T/{\log T} <t < T\ :\ 
|t-\gamma_c|< \delta(\epsilon)/\log T \text{ for some } \gamma_c\in Z(T) \right\}
\end{equation}
and then using Lemma~\ref{lem:sound} we get
\begin{align}
\frac{\delta(\epsilon)}{\log T} \sum_{\substack{\frac{T}{\log T}\le\gamma'\le T\\
	\left(\beta'-\frac{1}{2}\right)\log\gamma'\le\nu \\
	\gamma_c\in Z(T)}}
|M_{\gamma_c}|^{2k}
\ll \mathstrut & A^{2k} k^{2k} T\log^{2k}T +
   2^{2k} \int_{\frac{T}{\log T}}^T
      \left| 
\sum_{p<x^2}
   \frac{\Lambda_x(p)}{p^{\frac12 + \frac{1}{\log T} + i t}}
      \right|^{2k}dt\cr
\ll  \mathstrut& A^{2k} k^{2k} T\log^{2k}T 
+ 2^{2 k} k!\, T \left( \sum_{p<x^2}
   \frac{\Lambda_x(p)^2}{p^{1 + \frac{2}{\log T} }}
      \right)^{k} \cr
\ll  \mathstrut& A^{2k} k^{2k} T\log^{2k}T.
\end{align}
The last step used $\Lambda_x(p)\le \Lambda(p)$ and the fact that
\begin{equation}
\sum_{p\le x}\frac{\Lambda(p)^2}{p} \ll \log^2 x,
\end{equation}
which is a weak form of the prime number theorem.

Rearranging the above inequality and combining with (\ref{e:nui}), we have
\begin{equation}
\frac{e^{-C(\nu)}}{\nu^k}(1+O(\sqrt{\mathstrut\nu}))^{2k}T\log^{2k+1}T
\ll  A^{2k} k^{2k} T\log^{2k+1}T,
\end{equation}
which rearranges to give
\begin{equation}
\left(1+O(\sqrt{\mathstrut\nu})\right)^{2k}
\ll A^{2k} k^{2k} \nu^k e^{C(\nu)}.
\end{equation}
Letting $k=[1/\sqrt{A^2 e\nu}]$, we have 
a contradiction
if $\sqrt{\mathstrut\nu}C(\nu) \to 0$ as
as $\nu\to 0$. 
This completes the proof of Theorem~\ref{t:la}.

\section{Proofs of technical results}

In this section we provide the proofs of Proposition~\ref{p:logdea}
and Proposition~\ref{p:dnumber}.

\subsection{Proof of Proposition~\ref{p:dnumber}}\label{sec:dnumber}
A special case of the Proposition is the following:
\begin{claim} There exists $C_1>0$ such that the number of
$\gamma_n<T$ satisfying
\begin{equation}
N\left(\gamma_n+\frac{lC^*\logg T}{\log
T}\right)-N\left(\gamma_n+\frac{(lC^*-1)\logg T}{\log
T}\right)\le C_1\logg T,
\end{equation}
for all $|l|\le\log T/(C^*\logg T)$, is
\begin{equation}
\frac{T}{2\pi}\log T+O\left(\frac{T}{(\log T)^{m_0}}\right).
\end{equation}
Here $C_1$ is not depending on $C^*$.
\end{claim}

The proof of Claim follows easily from the same method below. Thus,
we omit the proof of it.

From now on, we are assuming that $\gamma_n$ satisfies $T/(\log
T)^{m_0+1}<\gamma_n<T$ and Claim. We recall
\begin{align}
\int_T^{T+H}|S(t+h)-S(t)|^{2k}dt=\mathstrut&\frac{H(2k)!}{(2\pi^2)^kk!}\log^k(2+h\log
T)\\
&+O\left(H(ck)^k\left(k+\log^{k-1/2}(2+h\log T)\right)\right)
\end{align}
uniformly for $T^a<H\le T$, $a>1/2$, $0<h<1$ and any positive
integer $k$, where $c$ is a positive constant and
$S(t)=\frac{1}{\pi}\arg\zeta(1/2+it)$. For this, see \cite[Theorem
4]{Ts}. Thus we have
\begin{equation}\label{e:tsang}
\int_0^T|S(t+h)-S(t)|^{2k}dt\ll T\left(Ak\right)^{2k},
\end{equation}
where $\log(2+h\log T)\ll k$. We note that
\begin{equation}\label{e:ost}
S(t+h)-S(t)=N(t+h)-N(t)-\frac{h}{2\pi}\log
t+O\left(\frac{h^2+1}{t}\right),
\end{equation}
where $N(t)$ is the number of zeros of $\zeta(s)$ in $0<\Im s<t$. By
this, we have
\begin{equation*}
\widetilde{S}(t,l_1,l_2)=N\left(t+\frac{(l_2-l_1)C^*\logg T}{\log
T}\right)-N(t)-\frac{(l_2-l_1)C^*\logg
T}{2\pi}+O\left(\frac{1}{t}\right),
\end{equation*}
where
\begin{equation}
\widetilde{S}(t,l_1,l_2)=S\left(t+\frac{(l_2-l_1)C^*\logg T}{\log
T}\right)-S(t)
\end{equation}
Using Claim, the last formula and (\ref{e:ost}), we have
\begin{align}
N(n,l_1,l_2)\le\mathstrut &\left|\widetilde{S}(t,l_1,l_2)\right|+\logg T\\
\le\mathstrut &\left|\widetilde{S}(t-h,l_1,l_2)\right|+3\logg T\\
&\mathstrut +\sum_{j=1}^{2}N\left(\gamma_n+\frac{l_jC^*\logg T}{\log
T}\right)-N\left(\gamma_n+\frac{(l_jC^*-1)\logg T}{\log T}\right)\\
\le\mathstrut &C_2\logg T+\left|\widetilde{S}(t-h,l_1,l_2)\right|
\end{align}
for $t=\gamma_n+l_1C^*\logg T/\log T$ and $0\le
h\le\logg T/\log T$, where $C_2=\max\{2C_1+3,A\}$. Using this,
we have
\begin{align}
\sum_{\substack{\frac{T}{(\log T)^{m_0+1}}<\gamma_n<T\\
N(n,l_1,l_2)\geqslant C\logg T}}(C\logg T)^{2k}
\mathstrut & \mathstrut \frac{\logg T}{\log
T}\\
\ll \mathstrut &T\log T(2C_2\logg
T)^{2k}+\sum_{\gamma_n<T}\int_{\gamma_n+\frac{(l_1C^*-1)\logg
T}{\log T}}^{\gamma_n+\frac{l_1C^*\logg T}{\log
T}}\left|2\widetilde{S}(t,l_1,l_2)\right|^{2k}dt\\
\ll \mathstrut & T\log T(2C_2\logg T)^{2k}
	+\log T\int_0^T\left|2\widetilde{S}(t,l_1,l_2)\right|^{2k}dt\\
\ll \mathstrut &T\log T\left((2C_2\logg T)^{2k}+(2C_2k)^{2k}\right)
\end{align}
for any sufficiently large $T$ and any $|l_1|,|l_2|\le\log
T/(C^*\logg T)$ with $0<l_2-l_1\le 2\log T/(C^*\logg T)$. We
put
\begin{equation}
k=[\logg T]\qquad\text{and}\qquad C=e^{m_0+2}(2C_2+1).
\end{equation}
By these and the last inequality, we have
\begin{equation*}
\sum_{\substack{|l_1|,|l_2|\le\frac{\log T}{C^*\logg
T}\\0<l_2-l_1\le\frac{2\log T}{C^*\logg T}}}
\sum_{\substack{\frac{T}{(\log T)^{m_0+1}}<\gamma_n<T\\
N(n,l_1,l_2)\geqslant C\logg T}}1\ll\frac{T(\log T)^4(2C_2\logg
T)^{2k}}{\left(C\logg T\right)^{2k}}\ll\frac{T}{(\log T)^{m_0}}.
\end{equation*}
We complete the proof of Proposition \ref{p:dnumber}.

\subsection{Proof of Proposition~\ref{p:logdea}}\label{sec:logdea}

We recall
\begin{equation}
\frac{\zeta'}{\zeta}(s)=O(\log
t)+\sum_{|\gamma-t|\le1}\frac{1}{s-\rho}
\end{equation}
holds uniformly for $t>1$ and $-2\le\R s\le1$. For this,
see \cite[Theorem 9.6 (A)]{Ti}. Using the last formula, it suffices
to show that the number of $\gamma_n$ in $T/(\log
T)^{m_0+1}\le\gamma_n<T$ such that $\gamma_n$ satisfies the
condition in Proposition \ref{p:dnumber} and
\begin{equation}
\sum_{\frac{C^*\logg T}{\log
T}<|\gamma_n-\gamma_m|\le1}\frac{1}{\gamma_n-\gamma_m}=O(\log
T)
\end{equation}
is
\begin{equation}
\frac{T}{2\pi}\log T+O\left(\frac{T}{(\log
T)^{m_0}}\right)\qquad(T\to\infty),
\end{equation}
because for $s=1/2+1/\log T+it$ and $|\gamma_n-t|\le A/\log
T$, we have
\begin{equation*}
\sum_{\frac{C^*\logg T}{\log
T}<|\gamma_n-\gamma_m|\le1}\frac{1}{s-\rho}-\frac{1}{i(\gamma_n-\gamma_m)}=O\left(\frac{1}{\log
T}\sum_{m=1}^{\infty}\frac{\logg T}{\left(\frac{m\logg T}{\log
T}\right)^2}\right)=O(\log T).
\end{equation*}
We recall that Proposition \ref{p:dnumber} implies
\begin{equation}
N\left(\gamma_n+\frac{lC^*\logg T}{\log T}\right)
=N(\gamma_n)+\frac{lC^*\logg T}{2\pi}+O\left(\logg T\right)
\end{equation}
for any integer $l$ with $|l|\le\log T/(C^*\logg T)$. This
immediately implies that for a sufficiently large $C^*>1$, we have
\begin{equation}
\max_{0\le k\le
N}|2\gamma_n-\gamma_{m_2+k}-\gamma_{m_1-k}|\frac{\log
T}{2\pi}=O(\logg T),
\end{equation}
where $\gamma_{m_1}$ is the greatest one in
$[\gamma_n-1,\gamma_n-C^*\logg T/\log T)$, $\gamma_{m_1}$ the least
one in $(\gamma_n+C^*\logg T/\log T,\gamma_n+1]$ and $N$ the largest
positive integer such that $\gamma_{m_1-N}$ and $\gamma_{m_2+N}$
belong to $[\gamma_n-C^*\logg T/\log T,\gamma_n+C^*\logg T/\log T]$.
By this and putting $ M(n)=\max_{0\le k\le
N}|2\gamma_n-\gamma_{m_2+k}-\gamma_{m_1-k}|,$ we have
\begin{equation}
M(n)\ll\frac{\logg T}{\log T}
\end{equation}
Using this and Proposition \ref{p:dnumber} and the fact
\cite[Theorems 9.3 and 14.13]{Ti} that the number of zeros between
$t$ and $t+1$ is
\begin{equation}
\frac{\log t}{2\pi}+O\left(\frac{\log t}{\logg
t}\right)\qquad(\text{as } t\to\infty),
\end{equation}
we have
\begin{align}
\sum_{\frac{C^*\logg T}{\log
T}<|\gamma_n-\gamma_m|\le1}\frac{1}{\gamma_n-\gamma_m}=\mathstrut&\sum_{0\le
k\le
N}\frac{2\gamma_n-\gamma_{m_2+k}-\gamma_{m_1-k}}{(\gamma_n-\gamma_{m_2+k})(\gamma_n-\gamma_{m_1-k})}+O(\log
T)\\
=\mathstrut&O\left(M(n)\sum_{k=1}^{\infty}\frac{\logg T}{\left(\frac{k\logg
T}{\log
T}\right)^2}\right)+O(\log T)\\
=\mathstrut&O(\log T).
\end{align}
Thus, we complete the proof of Proposition \ref{p:logdea}.

\end{document}